\documentclass{amsart}

\usepackage{latexsym}
\usepackage{amssymb}
\usepackage{amsmath,mathabx,mathtools}

\newtheorem{theorem}{Theorem}[section]
\newtheorem{lemma}[theorem]{Lemma}
\newtheorem{proposition}[theorem]{Proposition}

\theoremstyle{definition}
\newtheorem{definition}[theorem]{Definition}

\theoremstyle{remark}

\numberwithin{equation}{section}

\newcommand{\C}{\mathbb{C}}
\newcommand{\R}{\mathbb{R}}
\newcommand{\be}{\beta}
\newcommand{\lm}{\lambda}
\newcommand{\Gm}{\Gamma}
\newcommand{\re}{\operatorname{Re}}
\newcommand{\f}{\varphi}
\newcommand{\ep}{\varepsilon}

\newcommand{\Aut}{\operatorname{Aut}}

\newcommand{\Hh}{\mathcal H}
\newcommand{\K}{\mathcal K}
\newcommand{\erg}{\operatorname{Erg}}
\newcommand{\la}{\langle}
\newcommand{\ra}{\rangle}
\newcommand{\bb}{\bar B_n}

\newcommand{\mix}{\mathrm{mix}}
\newcommand{\wmix}{\mathrm{wmix}}
\newcommand{\bin}{\mathrm{bin}}
\newcommand{\wh}{\widehat}
\newcommand{\ol}{\overline}
\newcommand{\op}{\mathrm{op}}
\newcommand{\CP}{\mathrm{CP}}
\newcommand{\Span}{\mathrm{Span}}

\newcommand{\J}{\mathcal J}
\newcommand{\id}{\mathrm{id}}

\begin{document}

\title[Haagerup property and invariant states]{The Haagerup property for groups and for tracial von Neumann algebras in terms of invariant and mixing states}

\author{Paul Jolissaint}
\address{Universit\'e de Neuch\^atel,
       Institut de Math\'ematiques,       
       E.-Argand 11,
       2000 Neuch\^atel, Switzerland}
       
\email{pajolissaint@gmail.com}


\date{\today}


\begin{abstract}
The aim of the article is to provide two characterizations of the Haagerup property: one for locally compact, second countable groups, and the other one for finite von Neumann algebras.
Both are expressed in terms of approximations of some non-ergodic invariant states by mixing ones for actions on unital $C^*$-algebras on the one hand, and for pairs of tracial von Neumann algebras by mixing binormal states on the other hand.
\end{abstract}

\maketitle

\section{Introduction}

The present article aims to propose two characterizations of the Haagerup property: one for locally compact, second countable (lcsc) groups, and the other one for pairs of tracial von Neumann algebras (i.e. finite von Neumann algebras equiped with a faithful, normal tracial state).

In the first part of this article, $G$ denotes a lcsc group; we assume furthermore that it is non-compact because the compact case is not relevant to the Haagerup property \cite{ccjjv}. Recall that the latter property is often interpreted as
a strong negation of Kazhdan's property (T) (or as a weak form of amenability, but we won't need that interpretation here). Denoting by $P_{0,1}(G)$ the convex set of all continuous, normalized positive definite functions on $G$ that vanish at infinity, recall that $G$ has the \textit{Haagerup property} if there exists a sequence $(\f_n)_{n\geq 1}\subset P_{0,1}(G)$ which converges to $1$ uniformly on compact subsets of $G$.

Before presenting the motivation and the content of the first part of our article, we need to recall some definitions and fix some notation.

We consider here the following two types of actions of $G$:
\begin{itemize}
\item Continuous actions of $G$ on compact metrizable spaces; if $X$ is such a space and if $G$ acts continuously by homeomorphisms on $X$, we say that the pair $(X,G)$ is a $G$-\textit{dynamical system}, or simply a \textit{dynamical system} if $G$ is fixed, and we denote by $M(X)$ the compact convex set of all probability measures on $X$ equipped with the weak$^*$ topology, by $M_G(X)$ the closed convex subset of $G$-invariant elements of $M(X)$, by $\erg_G(X)$ the set of extreme points of $M_G(X)$, and finally by $M_{G,m}(X)$ the set of \textit{mixing} measures, i.e. elements $\mu\in M_G(X)$ such that, for all Borel sets $A,B\subset X$, one has
\[
\lim_{g\to\infty}\mu(gA\cap B)=\mu(A)\mu(B).
\]
\item Continuous actions of $G$ on separable, unital $C^*$-algebras: if $A$ is such a $C^*$-algebra, an \textit{action} of $G$ on $A$ is a homomorphism $\alpha:G\rightarrow \Aut(A)$ which is continuous in the sense that, for every $a\in A$, the map $g\mapsto \alpha_g(a)$ is norm continuous. We say that such a triple $(A,G,\alpha)$ is a $G$-$C^*$\textit{-dynamical system}, or simply a $C^*$\textit{-dynamical system} if $G$ is fixed.
We denote by $S(A)$ the convex compact set of all states on $A$ equipped with the weak$^*$ topology, by $S_G(A)$ the subset of all $G$-invariant ones, by $\erg_G(A)$ the set of all extreme points of $S_G(A)$ and by $S_{G,m}(A)$ the set of all \textit{mixing} states, i.e. the set of elements $\f\in S_G(A)$ such that 
\[
\lim_{g\to\infty}\f(\alpha_g(a)b)=\f(a)\f(b)
\]
for all $a,b\in A$. 
\end{itemize}
Observe that $M_G(X)$ (resp. $S_G(A)$) might be empty, and that
every dynamical system $(X,G)$ is a special case of a $C^*$-dynamical system since the former corresponds to $A=C(X)$, the commutative $C^*$-algebra of continuous functions on $X$ equipped with the action of $G$ by left translations: $(g\cdot a)(x)=a(g^{-1}x)$ for all $a\in C(X)$, $g\in G$ and $x\in X$, and due to Riesz representation of states on $C(X)$ by regular probability measures on $X$.

In the remarkable article \cite{GW}, E. Glasner and B. Weiss obtained an outstanding characterization of property (T) for a locally compact, second countable group $G$. We reproduce (not \textit{verbatim}) A. Valette's summary of the main results of \cite{GW} as it appears in \cite[Section 7.5.2]{ccjjv}:
\textit{The following conditions are equivalent:}
\textit{\begin{enumerate}
\item [(1)] $G$ has property (T);
\item [(2)] for every dynamical system $(X,G)$ with $M_G(X)\not=\emptyset$, $\erg_G(X)$ is closed in $M_G(X)$;
\item [(3)] let us denote by $\Sigma$ the metrizable compact set of all closed subsets of the one-point compactification $G^+=G\cup\{\infty\}$ containing $\infty$; then $\erg_G(\Sigma)$ is closed in $M_G(\Sigma)$;
\item [(4)] the set $\erg_G(\Sigma)$ is not dense in $M_G(\Sigma)$.
\end{enumerate}}
Then, A. Valette wrote: \textit{According to the philosophy that, to any characterization of property (T) there is a parallel characterization of the Haagerup property, there should be a definition of the Haagerup property corresponding to the above definition of property (T). What is it?}

The first aim of the present note is to propose such a characterization, which is our first main result. It should be pointed out that essentially a proof of the same result for countable groups is sketched and published in \cite[Theorem 13.21]{G}. Nevertheless, we believe that a detailed proof at this level of generality, in particular the $C^*$-algebraic aspect, is worth publishing.

\begin{theorem}\label{HAP1}
Let $G$ be a locally compact, second countable group. 
\begin{enumerate}
\item [(1)] Suppose that there exists a $C^*$-dynamical system $(A,G,\alpha)$ which has the follownig property: there exists $\f\in S_G(A)\setminus \erg_G(A)$ and $(\f_n)_{n\geq 1}\subset S_{G,m}(A)$ such that $\f=w^*-\lim_n \f_n$. Then $G$ has the Haagerup property. 
\item [(2)] Conversely, if $G$ has the Haagerup property, there exists a non-ergodic measure $\nu\in M_G(\Sigma)$ and a sequence $(\eta_n)_{n\geq 1}\subset M_{G,m}(\Sigma)$ such that $\nu=w^*-\lim_{n}\eta_n$. More precisely, $M_{G,m}(\Sigma)$ contains more than one element and its weak$^*$ closure is convex.
\end{enumerate}
\end{theorem}
The referee made the following observation: if $M_{G,m}(\Sigma)$ is weak$^*$-dense in $M_G(\Sigma)$, then Theorem 1.1(1) implies that $G$ has the Haagerup property. However, we do not know whether the converse holds or not.

The proof of Theorem \ref{HAP1} is largely inspired by the techniques used in \cite{GW} and it is contained in Section 3. Section 2 contains material needed in the latter section.

As is well known, the Haagerup property for (finite) von Neumann algebras and for pairs of such algebras plays also an important role: see e.g. \cite{Popa} and \cite{OOT}. Even though it was necessary to define and study it for arbitrary von Neumann algebras, and in particular for applications to quantum group von Neumann algebras (which are often infinite), it is basically a property of tracial von Neumann algebras. Indeed, among the several equivalent definitions, in the end they are all equivalent to that the finite von Neumann algebra $p(M\rtimes_\sigma\R)p$ has the Haagerup property where $\sigma$ is a modular action and $p$ is a finite projection with central support $z(p)=1$. See for instance \cite[page 82]{OOT}, \cite{CS} and \cite{OT}.

Recall e.g. from \cite{Jol} that a tracial von Neumann algebra $(M,\tau)$ has the \textit{Haagerup property} if there exists a net $(\phi_i)_{i\in I}$ of completely positive maps from $M$ to itself with the following properties:
\begin{enumerate}
\item [(a)] one has $\tau\circ\phi_i\leq \tau$ ($\phi_i$ is \textit{subtracial}), $\phi_i(1)\leq 1$ ($\phi_i$ is \textit{subunital}) and $\phi_i$ extends to a compact operator on $L^2(M,\tau)$ for every $i\in I$;
\item [(b)] for every $x\in M$, one has $\lim_i\Vert \phi_i(x)-x\Vert_2=0$.
\end{enumerate}

The latter property is independent of the chosen tracial state (\cite[Proposition 2.4]{Jol}), and it admits equivalent conditions via \textit{$M$-bimodules} (e.g. \cite[Theorem 3.4]{BF}, \cite[Theorem 9]{OOT}) : given a von Neumann algebra $M$, such a module is a Hilbert space $\Hh$ equipped with a (left) normal representation of $M$, denoted by $(x,\xi)\mapsto x\xi$, and with a commuting normal action of the opposite algebra $M^\op$ denoted by the right action $(y^\op,\xi)\mapsto \xi y$.

In the breakthrough article \cite{Popa}, S. Popa uses  the Haagerup property for suitable discrete groups to prove for the first time the existence of II$_1$ factors with trivial fundamental group, which solved a long standing open problem posed by R. V. Kadison. In fact, he needed  a relative version of the Haagerup property for pairs of finite von Neumann algebras that will be presented in details in the Section 4.

In this introductory section, we describe our next result only in the case of a single tracial von Neumann algebra $(M,\tau)$, the case of pairs $B\subset M$ being treated in Section 4.

Thus, given such a tracial von Neumann algebra, our next result is a characterization of the Haagerup property for $M$ in terms of the set of binormal states $\bin(M)$ on the algebraic tensor product $M\otimes M^\op$. 

Before recalling the latter notion, let us motivate our approach by the following observation in the framework of lcsc groups: given such a group $G$, recall from \cite{BR} that a unitary representation $(\pi,\Hh)$ of $G$ is \textit{mixing} (or a $C_0$-\textit{representation}) if all its coefficient functions $g\mapsto \la \pi(g)\xi|\eta\ra$ ($\xi,\eta\in \Hh)$ vanish at infinity. In particular, all elements of $P_{0,1}(G)$ are coefficients of mixing representations. 

Recall also that the convex set $P_1(G)$ of all continuous, normalized positive definite functions on $G$ identifies with the state space of the maximal $C^*$-algebra $C^*(G)$ of $G$ (\cite[Chapter 13]{DiC}), and it follows from the definition of the Haagerup property that $G$ possesses this property if and only if $P_{0,1}(G)$ is weak*-dense in $P_1(G)$ as sets of states of $C^*(G)$ (see e.g. \cite[Proposition 3.1]{JolC0}). Coming from mixing representations of $G$, it is natural to call elements of $P_{0,1}(G)$ \textit{mixing} normalized positive definite functions on $G$.

Moreover, if $G$ is discrete, every element $\psi$ of $P_{1}(G)$ gives rise to a completely positive map $m_\psi$ on the von Neumann algebra $L(G)$ associated to $G$ and generated by the left regular representation $\lm$, so that $m_\psi(\lm(g))=\psi(g)\lm(g)$, $g\in G$. Then $m_\psi$ extends to the bounded linear operator of multiplication by $\psi$ on $\ell^2(G)$ which is compact if and only if $\psi\in P_{0,1}(G)$, i.e. if and only if the associated representation $\pi_\psi$ is mixing in the above sense.

Following \cite[Page 10]{EL}, a state (i.e. a positive, normalized linear form) on the algebraic tensor product $M\otimes M^\op$ is \textit{binormal} if it is separately normal; there is then a natural correspondence between $\bin(M)$ and the set of unit vectors of $M$-bimodules; more precisely:
\begin{itemize}
\item  to every such vector $\xi$, one associates the element $\f_\xi\in \bin(M)$ defined by $\f_\xi(x\otimes y^\op)=\la x\xi y|\xi\ra$, $x,y\in M$;
\item conversely, if $\f\in \bin(M)$, the GNS triple $(\pi_\f,\Hh_\f,\xi_\f)$ applied to the pair $(M\otimes M^\op, \f)$
yields the pointed $M$-bimodule $(\Hh_\f,\xi_\f)$ such that $\f(x\otimes y^\op)=\la x\xi_\f y|\xi_\f\ra$; moreover, $\xi_\f$ is \textit{cyclic} in the following sense: the set $\{x\xi_\f y\colon x,y\in M\}$ is total in $\Hh_\f$. 
\end{itemize}
For future use, if $\Hh$ is a $M$-bimodule and $\xi\in \Hh$, we denote by $\Span(M\xi M)$ the set of all finite sums $\sum _i x_i \xi y_i$ with $x_i,y_i\in M$.

We also observe that $\bin(M)$ is weak*-dense in the state space of the tensor product $C^*$-algebra $M\otimes_\bin M^\op$ which is the completion of $M\otimes M^\op$ with respect to the $C^*$-norm
\[
\Vert x\Vert_\bin=\sup\{\Vert \pi_\f(x)\Vert\colon \f\in \bin(M\otimes_\bin M^\op)\}
\]
For all these facts, see \cite{EL}.

In order to state the next theorem, we briefly describe the subset of what we call \textit{mixing} binormal states on $M$; see Section 4 for detailed definitions. Let $\f\in \bin(M)$ and let $(\Hh_\f,\xi_\f)$ be the associated pointed $M$-bimodule. Then $\xi_\f$ is said to be \textit{(two-sided) bounded} if there exists a constant $c>0$ such that
\[
\Vert x\xi_\f\Vert\leq c\Vert x\Vert_2\quad \textrm{and}\quad \Vert \xi_\f x\Vert \leq c\Vert x\Vert_2\quad (x\in M).
\]
It turns out that such a vector yields a normal, completely positive map $\phi_\f:M\rightarrow M$ characterized by 
\[
\tau(\phi_\f(x)y)=\la x\xi_\f y|\xi_\f\ra \quad (x,y\in M).
\]
Boundedness of $\xi_\f$ implies moreover that $\phi_\f$ extends to a bounded linear operator $\wh{\phi_\f}$ on $L^2(M,\tau)$; then $\phi_\f$ is called \textit{$L^2$-compact} if $\wh{\phi_\f}$ is a compact operator.

Finally, $\f\in \bin(M)$ is called \textit{mixing} if $\xi_\f$ is bounded and if $\wh{\phi_\f}$ is $L^2$-compact. We denote by $\bin_\mix(M)$ the set of all mixing binormal states on $M$.

Our second result is then the following promised characterization of the Haagerup property for finite von Neumann algebras which is proved in Section 4.

\begin{theorem}\label{HAP2}
Let $(M,\tau)$ be a tracial von Neumann algebra. Then $M$ has the Haagerup property if and only if $\bin_\mix(M)$ is dense in $\bin(M)$.
\end{theorem}

Our final result is motivated by the following question posed by the referee: Since Theorem 1.1 is inspired by Glasner-Weiss Theorem on property (T), does Theorem 1.2 have an analogue in the context of property (T) for von Neumann algebras? 

Using \cite[Theorem 5.1]{JolT} and \cite[Theorem 3.10]{JolWMAP}, we prove in Section 5 a characterization of property (T) for type II$_1$ factors (the latter property will be recalled in Section 5) based on what we call weakly mixing binormal states. More precisely, let $\f\in \bin(M)$ and let $(\Hh_\f,\xi_\f)$ be its associated pointed $M$-bimodule. Let $\rho_\f:U(M)\rightarrow U(\Hh_\f)$ be the unitary representation defined by
\[
\rho_\f(u)\xi=u\xi u^*\quad (u\in U(M),\xi\in \Hh_\f).
\]
Then we say that $\f$ is \textit{weakly mixing} if the unitary representation $\rho_\f$ is weakly mixing, i.e. it has no non-zero finite-dimensional subrepresentation. See for instance \cite{BR}. We denote by $\bin_\wmix(M)$ the set of all weakly mixing binormal states on $M$. 

Let us fix one more notation before stating our last result. Let $(M,\tau)$ be a tracial von Neumann algebra; we denote by $\f_\tau$ the following binormal state on $M$: $\f_\tau(x\otimes y^\op)=\tau(xy)$ for all $x,y\in M$. It is easy to check that its associated completely positive map is $\id_M$.

Here is our last result.

\begin{theorem}\label{propT}
Let $M$ be a type $\mathrm{II}_1$ factor and let $\tau$ be its tracial state. Then $M$ has property (T) if and only if $\f_\tau$ does not belong to the weak$^*$ closure of $\bin_\wmix(M)$.
\end{theorem}

\section{Prerequisites for the proof of Theorem \ref{HAP1}}

\subsection{Actions on $C^*$-algebras and their invariant states}

In this section, we recall properties of ($C^*$-)dynamical systems that will be useful in the proof of Theorem \ref{HAP1}. Some results should be well known, but for the sake of clarity, we provide some proofs for the reader's convenience.

Let $(A,G,\alpha)$ be a unital $C^*$-dynamical system. We always assume that $S_G(A)$ is non-empty. Hence, $\erg_G(A)$ is non-empty either, as the former is the weak$^*$ closed convex hull of the latter by Krein-Milman Theorem.

Given $\f\in S_G(A)$, we denote by $(\pi_\f,H_\f,\xi_\f,u_\f)$ the Gel'fand-Naimark-Segal (GNS) construction associated to the pair $(A,\f)$: $\pi_\f:A\rightarrow B(H_\f)$ is a $*$-representation of $A$ on $H_\f$, $\xi_\f\in H_\f$ is a cyclic, unit vector such that
\[
\f(x)=\langle\pi_\f(x)\xi_\f|\xi_\f\rangle\quad (x\in A)
\]
and $u_\f$ is a unitary representation of $G$ on $H_\f$ characterized by
\[
u_\f(g)\pi_\f(x)\xi_\f=\pi_\f(\alpha_g(x))\xi_\f\quad (g\in G, \ x\in A).
\]
It implements the action $\alpha$ on $H_\f$ in the sense that
\[
u_\f(g)\pi_\f(x)u_\f(g^{-1})=\pi_\f(\alpha_g(x))\quad (g\in G,\ x\in A).
\]

Recall that, for a given dynamical system $(X,G)$, every element $\mu\in M_G(X)$ corresponds to a state in $S_G(C(X))$ still denoted by $\mu$ and defined by
\[
\mu(a)=\int_X a(x)d\mu(x)\quad (a\in C(X), x\in X).
\]
Then the GNS construction yields $H_\mu=L^2(X,\mu)$, $\xi_\mu=1_X$,  $\pi_\mu(a)\xi=a\cdot\xi$ for all $a\in C(X)$ and $\xi\in L^2(X,\mu)$ and $(u_\mu(g)\xi)(x)=\xi(g^{-1}x)$ for all $g\in G$, $\xi\in L^2(X,\mu)$ and $x\in X$.

For further use, we observe that, as $\xi_\mu=1_X$ is a cyclic vector for $\pi_\mu(C(X))$,
\[
\pi_\mu(C(X))'\cap u_\mu(G)'=L^\infty(X,\mu)\cap u_\mu(G)'=\{a\in L^\infty(X,\mu)\colon g\cdot a=a\ \forall g\in G\}
\]
is the space of $G$-fixed elements of $L^\infty(X,\mu)$.
See for instance \cite[Proposition 3.4.2]{GKP} for the equality $\pi_\mu(C(X))'=L^\infty(X,\mu)$. This means that $\mu$ is ergodic in the usual sense (i.e.  $\mu(A)(1-\mu(A))=0$ for every $G$-invariant Borel set $A$) if and only if $\pi_\mu(C(X))'\cap u_\mu(G)'=\C 1$. For general $C^*$-dynamical systems, one has (see for instance \cite[Exercise I.9.7]{TakI}):

\begin{proposition}\label{ergCstar}
Let $(A,G,\alpha)$ be a $C^*$-dynamical system such that $S_G(A)\not=\emptyset$, and let $\f\in S_G(A)$. Then $\f\in \erg_G(A)$ if and only if $\pi_\f(A)'\cap u_\f(G)'=\C 1$.
\end{proposition}

We observe that, as $u_\f(g)\xi_\f=\xi_\f$ for every $g\in G$, the orthogonal subspace $(\C\xi_\f)^\perp\eqqcolon H_\f^0$ is $G$-invariant; we denote by $u_\f^0$ the restriction of $u_\f$ to $H_\f^0$.

\begin{proposition}\label{ergCompl}
Let $(A,G,\alpha)$ be a $C^*$-dynamical system such that $S_G(A)\not=\emptyset$.
\begin{enumerate}
\item [(1)] If $\f\in S_G(A)\setminus \erg_G(A)$, then there exists $0\not=\xi\in H_\f^0$ such that $u_\f(g)\xi=\xi$ for every $g\in G$.
\item [(2)] For $\f\in S_G(A)$, one has $\f\in S_{G,m}(A)$ if and only if $u_\f^0$ is a mixing representation, i.e. for all $\xi,\eta\in H_\f^0$, the associated coefficient function $g\mapsto \la u_\f^0(g)\xi|\eta\ra$ belongs to $C_0(G)$.
\item [(3)] One has $S_{G,m}(A)\subset \erg_G(A)$.
\item [(4)] If $(B,G,\beta)$ is a $C^*$-dynamical system such that there exists a unital, $G$-equivariant $*$-homomorphism $\theta:B\rightarrow A$, then $\theta_*(S_{G,m}(A))\subset S_{G,m}(B)$, where $\theta_*$ is the dual map $\theta_*(\psi)=\psi\circ\theta$ for every $\psi\in A^*$.
\end{enumerate}
\end{proposition}
\begin{proof}
(1) By Proposition \ref{ergCstar}, there exists a projection $p\in \pi_\f(A)'\cap u_\f(G)'$ such that $p\not=0,1$. We claim then that $0\not=p\xi_\f\not=\xi_\f$. Indeed, for instance if $p\xi_\f=0$, then we would have $p\pi_\f(x)\xi_\f=\pi_\f(x)p\xi_\f=0$ for every $x\in A$, implying that $p=0$ by cyclicity of $\xi_\f$. Then set $\xi=p\xi_\f-\la p\xi_\f|\xi_\f\ra\xi_\f$, which is orthogonal to $\xi_\f$ and $u_\f^0(G)$-invariant. We claim that $\xi\not=0$. Otherwise, we would have
$p\xi_\f=\la \xi|\xi_\f\ra\xi_\f=\Vert p\xi_\f\Vert^2\xi_\f$ hence
\[
\frac{p}{\Vert p\xi_\f\Vert^2}\pi_\f(x)\xi_\f=\pi_\f(x)\xi_\f\quad (x\in A)
\]
and $p=\Vert p\xi_\f\Vert^2 1$ but, as $p^2=p$, this would imply that $\Vert p\xi_\f\Vert=0$ or $1$, which is impossible by the above observations.\\
(2) Set $E\coloneqq \{\pi_\f(x)\xi_\f-\f(x)\xi_\f\colon x\in A\}$. Then it is a dense subspace of $H_\f^0$, as is easily verified. Next, we have for all $a,b\in A$
\[
\la u_\f^0(g)[\pi_\f(a)\xi_\f-\f(a)\xi_\f]|\pi_\f(b)\xi_\f-\f(b)\xi_\f\ra=
\f(b^*\alpha_g (a))-\f(b^*)\f(a)\to 0
\]
as ${g\to\infty}$ if and only if $\f\in S_{G,m}(A)$, and this proves the second assertion of the proposition.\\
(3) Applying (1), if there existed $\f \in S_{G,m}(A)$ such that $\f\notin \erg_G(A)$, there would exist $\xi\not=0$ such that $\xi\perp\xi_\f$ and $u_\f(g)\xi=\xi$ for all $g\in G$. But then
\[
\Vert \xi\Vert^2=\lim_{g\to\infty} \la u_\f(g)\xi|\xi\ra=0
\]
which is a contradiction.\\
(4) One has for all $\f\in S_{G,m}(A)$ and $x,y\in B$:
\begin{multline*}
\theta_*(\f)(\beta_g(x)y)=\f(\theta(\beta_g(x))\theta(y))=\f(\alpha_g(\theta(x))\theta(y))\\
\to_{g\to\infty} \f(\theta(x))\f(\theta(y))=\theta_*(\f)(x)\theta_*(\f)(y).
\end{multline*}
This shows that $\theta_*(\f)\in S_{G,m}(B)$.
\end{proof}

Proofs of some results in \cite{GW} require the notion of infinite product of probability spaces. Here, these are replaced by infinite (minimal) tensor products of unital $C^*$-algebras, as presented in \cite[Section XIV.1]{Tak3}.

Let $A$ be a separable, unital $C^*$-algebra and $\phi\coloneqq (\f_n)_{n\geq 1}\subset S(A)$. We denote by $B=\bigotimes_{n\geq 1} A_n$ the infinite tensor product $C^*$-algebra where $A_n=A$ for $n$. It is the inductive limit $C^*$-algebra of the sequence $B_1\subset B_2\subset\cdots \subset B_n\cdots$ where $B_1=A$, $B_2=A\otimes_{\min} A,\ldots$ and where $B_n$ identifies to $B_n\otimes 1$ in $B_{n+1}$ for every $n$. One associates to $\phi$ the following state on $B$ denoted by $\underset{{n\geq 1}}{\otimes}\f_n$ and defined by
\[
\underset{{n\geq 1}}{\otimes}\f_n(x_1\otimes\cdots \otimes x_n\otimes 1\otimes \cdots)=\f_1(x_1)\cdots \f_n(x_n)
\]
for all $n\geq 1$ and $x_1,\ldots,x_n\in A$.

If $(A,G,\alpha)$ is a $C^*$-dynamical system, then so is $(B,G,\beta)$, where $\beta=\underset{{n\geq 1}}{\otimes}\alpha$ is the action of $G$ defined by
\[
\beta_g(x_1\otimes\cdots \otimes x_n\otimes 1\otimes \cdots)=\alpha_g(x_1)\otimes \cdots\otimes \alpha_g(x_n)\otimes\cdots
\]
for all $g\in G$, $n\geq 1$ and $x_1,\ldots,x_n\in A$. We remark that if $\phi\subset S_G(A)$ then $\underset{n\geq 1}{\otimes}\f_n\in S_G(B)$.

Finally, let us denote by $J_G(B,\phi)$ the set of all states (\textit{joinings}) $\f\in S_G(B)$ such that one has for all $n\geq 1$ and $a\in A$:
\[
\f(1\otimes\cdots\otimes \underset{n}{a}\otimes 1\otimes\cdots)=\f_n(a)
\]
where $\underset{n}{a}$ means that $a$ is in the $n$th position in the tensor product. 

It is clear that $J_G(B,\phi)$ is convex and compact, and it is non-empty since $\underset{n\geq 1}{\otimes}\f_n\in J_G(B,\phi)$.

\begin{lemma}\label{lemJ}
Let us keep notations and definitions above.
\begin{enumerate}
\item [(1)] If $\phi\subset \erg_G(A)$ and if $\f\in J_G(B,\phi)$ is extremal in $J_G(B,\phi)$, then it belongs to $\erg_G(B)$.
\item [(2)] If $\phi\subset S_{G,m}(A)$ then $\underset{n\geq 1}{\otimes}\f_n\in S_{G,m}(B)$.
\end{enumerate}
\end{lemma}
\begin{proof}
(1) Let us write $\f=t\omega+(1-t)\psi$ with $\omega,\psi\in S_G(B)$. Then
\begin{multline*}
\f(1\otimes\cdots\otimes \underset{n}{a}\otimes 1\otimes\cdots)=\f_n(a)=\\
t\omega(1\otimes\cdots\otimes \underset{n}{a}\otimes 1\otimes\cdots)+
(1-t)\psi(1\otimes\cdots\otimes \underset{n}{a}\otimes 1\otimes\cdots)
\end{multline*}
for every $n$ and every $a\in A$, hence, since each $\f_n\in \erg_G(A)$, one has 
\[
\omega(1\otimes\cdots\otimes \underset{n}{a}\otimes 1\otimes\cdots)
=\psi(1\otimes\cdots\otimes \underset{n}{a}\otimes 1\otimes\cdots)=\f_n(a)
\]
for all $n$ and $a$, thus $\omega,\psi\in J_G(B,\phi)$, and as $\f$ is extremal in $J_G(B,\phi)$, we get $\omega=\psi=\f$.\\
(2) If $\f_n\in S_{G,m}(A)$ for every $n$, it is straightforward to check that, for every $n$ and for all $x_1,\ldots,x_n,y_1,\ldots,y_n\in A$ one has
\[
\lim_{g\to\infty}\f_1(\alpha_g(x_1)y_1)\cdots \f_n(\alpha_g(x_n)y_n)=\f_1(x_1)\cdots \f_n(x_n)\f_1(y_1)\cdots \f_n(y_n).
\]
By density, $\underset{n\geq 1}{\otimes}\f_n\in S_{G,m}(B)$.
\end{proof}

\subsection{The compact metrizable space $\Sigma$}

In order to prove statement (2) of Theorem \ref{HAP1}, we need to recall some properties of the space $\Sigma$ as it appears in the latter theorem. We feel necessary to give more details than in \cite{GW} on the dynamical system $(\Sigma,G)$. 

Let $G^+=G\cup\{\infty\}$ denote the one-point compactification of the locally compact, second countable group $G$, and let $\Sigma$ be the set of all closed subsets of $G^+$ containing $\infty$. Then it is easy to verify that, if $\infty\in Y\subset G^+$, then $Y\in \Sigma$ if and only if $Y\setminus\{\infty\}$ is a closed subset of $G$. The \textit{Hausdorff topology} (also called the \textit{Chabauty topology}) on $\Sigma$ is generated by the sets of the form $U(C;V_1,\ldots,V_n)$ where $C\subset G^+$ is compact, $V_1,\ldots,V_n\subset G^+$ are open, and where
\[
U(C;V_1,\ldots,V_n)\coloneqq \{Y\in \Sigma\colon Y\cap C=\emptyset,\ Y\cap V_j\not=\emptyset\ \forall j=1,\ldots,n\}.
\]
It turns out that $\Sigma$ is a compact, metrizable space on which $G$ acts by left translations, with the convention that $g\,\infty=\infty$ for every $g\in G$.
Thus, $(\Sigma,G)$ is a dynamical system. For future use, if $S\subset G$ is any set, we put $S^+\coloneqq S\cup\{\infty\}$.

Consider now a dynamical system $(X,G)$ such that $M_{G,m}(X)\not=\emptyset$. For every $\mu\in M_{G,m}(X)$ and every closed, non empty set $C\subset X$, the map $\theta_C:X\rightarrow \Sigma$ defined by
\[
\theta_C(x)=\{g\in G\colon g^{-1}x\in C\}^+ \quad (x\in X)
\]
is easily seen to be $G$-equivariant and continuous. Hence, using Proposition \ref{ergCompl} (4), we see that the associated measure $\theta_C(\mu)$ on $\Sigma$ defined by $\mu\circ \theta_C^{-1}$ belongs to $M_{G,m}(\Sigma)$. Moreover, it is easy to check that $\theta_C^{-1}(S)=C$, where $S$ denotes the singleton set $S=\{\{e\}^+\}\subset \Sigma$,
hence 
\[
\theta_C(\mu)(S)=\mu(C).
\]
This implies that, if $C_1,C_2$ are non empty closed subsets of $X$ such that $\mu(C_1)\not=\mu(C_2)$, then $\theta_{C_1}(\mu)\not=\theta_{C_2}(\mu)$.

We will present a detailed proof of statement (2) of Theorem \ref{HAP1} even if it is very similar to the ones of \cite[Proposition 2' and Theorem 2']{GW}; it rests on Lemmas 1 and 2 in \cite{GW} that we recall now.

In order to do that, and for the rest of the article, we fix an increasing sequence of symmetric, compact sets $K_n=K_n^{-1}\nearrow G$ where $e\in \mathring{K}_1$ and $K_n\subset \mathring{K}_{n+1}$ for every $n$.

\begin{lemma}\label{lemma1}
\cite[Lemma 1]{GW}
Let $f\in C(\Sigma)$; then
\begin{enumerate}
\item [(1)] For every $E\in \Sigma$,
\[
f(E)=\lim_n f(E\cap K_n^+).
\]
\item [(2)] Given $\ep>0$, there exists $n\geq 1$ such that for all $E,E'\in \Sigma$ with $E\cap K_n=E'\cap K_n$, one has $|f(E)-f(E')|<\ep$.
\end{enumerate}
\end{lemma}

\begin{lemma}\label{lemma2}
\cite[Lemma 2]{GW}
Let $G$ be a locally compact, second countable group, let $(X,G)$ be a dynamical system which admits a measure $\mu\in M_G(X)$ with the following property: There exists a sequence $(A_n)_{n\geq 1}$ of Borel subsets of $X$ such that $\mu(A_n)=1/2$ for every $n$ and for every $g\in G$
\[
\lim_{n\to\infty}\mu(gA_n\bigtriangleup A_n)=0.
\]
Then there exist two sequences $(B_n)$ and $(B_n')$ of Borel subsets of $X$ such that $B_n\cap B_n'=\emptyset$ for every $n$, $\lim_n\mu(B_n)=\lim_n\mu(B_n')=1/2$ and, for every compact set $K\subset G$,
\[
\lim_n \mu\Big(\bigcap_{k\in K} k^{-1}B_n\Big)=1/2.
\]
Furthermore, for all $x\in X$ and $n\geq 1$, the two sets
\begin{equation*}
F(x,n)\coloneqq \{g\in G\colon g^{-1}x\in B_n\} \quad \textrm{and}\quad
F'(x,n)\coloneqq \{g\in G\colon g^{-1}x\in B_n'\}
\end{equation*}
are closed subsets of $G$.
\end{lemma}

\section{Proof of Theorem \ref{HAP1}}

Let $G$ be a locally compact, second countable group. 
\par
\noindent
(1) Let us assume that there is a $C^*$-dynamical system $(A,G,\alpha)$ for which there exists a $G$-invariant, non-ergodic state $\f$ on $A$ and a sequence $(\f_n)_{n\geq 1}\subset S_{G,m}(A)$ such that $\f=w^*-\lim_n \f_n$. Recall that $P_{0,1}(G)$ is the set of all continuous, normalized, positive definite functions on $G$ which belong to $C_0(G)$. Let us fix a number $\ep>0$ and a compact set $K\subset G$. 
We have to prove the existence of some $\rho\in P_{0,1}(G)$ such that
\begin{equation}\label{inegHAP}
\sup_{g\in K}|\rho(g)-1|\leq \ep.
\end{equation}
By Proposition \ref{ergCompl}, as $\f$ is not ergodic, there exists $\xi\in H_\f^0$, $\Vert \xi\Vert=1$, such that $u_\f^0(g)\xi=\xi$ for every $g\in G$. There also exists $a\in A$ such that $\f(a)=0$, $\Vert\pi_\f(a)\xi_\f-\xi\Vert\leq \ep/4$ and $\Vert \pi_\f(a)\xi_\f\Vert=1$. 
One has
\begin{multline*}
\sup_{g\in G}|\f(a^*\alpha_g(a))-1|=
\sup_{g\in G}|\la u_\f(g)\pi_\f(a)\xi_\f|\pi_\f(a)\xi_\f\ra-1|\\
\leq 
\sup_{g\in G}|\la u_\f(g)(\pi_\f(a)\xi_\f-\xi)|\pi_\f(a)\xi_\f\ra|+ |\la \xi|\pi_\f(a)\xi_\f-\xi\ra|
\leq \frac{\ep}{2}.
\end{multline*}
Set $B=\bigotimes_{n\geq 1}A_n$ with $A_n=A$ for every $n$, and $\phi=\underset{n\geq 1}{\otimes}\f_n$. By Lemma \ref{lemJ}, $\phi\in S_{G,m}(B)$. Hence in particular, $u_\phi^0$ is a $C_0$ representation. For every $n\geq 1$, let $a_n=1\otimes\cdots\otimes\underset{n}{a}\otimes 1\otimes\cdots$ so that
\[
\psi_n(g)\coloneqq
\la u_\phi(g)\pi_\phi(a_n)\xi_\phi|\pi_\phi(a_n)\xi_\phi\ra=\f_n(a^*\alpha_g(a))\to_{n\to\infty}\f(a^*\alpha_g(a))\eqqcolon \psi(g)
\]
for every $g\in G$. Then $\psi_n\in P_{0,1}(G)$, and by Lebesgue Dominated Convergence Theorem, 
\[
\lim_{n\to\infty}\int_Gh(g)\psi_n(g)dg=\int_Gh(g)\psi(g)dg
\]
for every $h\in L^1(G)$, and by \cite[Theorem 13.5.2]{DiC}, the sequence $(\psi_n)_{n\geq 1}$ converges to $\psi$ uniformly on compact subsets of $G$. Thus, there exists $N>0$ such that, for every $n\geq N$, one has
\[
\sup_{g\in K}|\psi_n(g)-\psi(g)|\leq \frac{\ep}{2}.
\]
Finally, for every $n\geq N$, one has 
\[
\sup_{g\in K}|\psi_n(g)-1|\leq \sup_{g\in K}|\psi_n(g)-\psi(g)|+\sup_{g\in G}|\f(a^*\alpha_g(a))-1|\leq \ep.
\]
Setting $\rho=\psi_n$ for some $n\geq N$, we get (\ref{inegHAP}), and $G$ has the Haagerup property.

\par\noindent
(2) Assume next that $G$ has the latter property. By \cite[Theorem 2.2.2]{ccjjv} and by \cite[Lemma 1.3]{AEG}, there exists a dynamical system $(X,G)$ equipped with a probability measure $\mu\in M_{G,m}(X)$, whose support is $X$, and a sequence $(A_n)_{n\geq 1}$ of Borel subsets of $X$ such that $\mu(A_n)=1/2$ for every $n$ and, for every compact set $K\subset G$, 
\[
\lim_{n\to\infty}\sup_{g\in K}\mu(gA_n\bigtriangleup A_n)=0.
\]
By Lemma \ref{lemma2}, there exists two sequences $(B_n)_{n\geq 1}$ and $(B_n')_{n\geq 1}$ of Borel subsets of $X$ such that $B_n\cap B_n'=\emptyset$ for every $n$,
\[
\lim_n \mu(B_n)=\lim_n\mu(B_n')=1/2,
\]
and 
\[
\lim_n \mu\Big(\bigcap_{k\in K}k^{-1}B_n\Big)=1/2
\]
for every compact set $K\subset G$. Moreover, for all $x\in X$ and $n\geq 1$, the two sets 
\begin{equation}\label{F(x,n)}
F(x,n)\coloneqq \{g\in G\colon g^{-1}x\in B_n\} \quad \textrm{and}\quad
F'(x,n)\coloneqq \{g\in G\colon g^{-1}x\in B_n'\}
\end{equation}
are closed subsets of $G$.

In order to prove that $M_{G,m}(\Sigma)$ contains more than one element, choose two non empty closed sets $C_1,C_2\subset X$ such that $\mu(C_1)\not=\mu(C_2)$; then the two probability measures $\theta_{C_1}(\mu)$ and $\theta_{C_2}(\mu)$ as defined in Section 3 belong to $M_{G,m}(\Sigma)$ and are distinct. 

Now, let 
$\nu_1,\nu_2\in M_{G,m}(\Sigma)$; then $\sigma\coloneqq \mu\times\nu_1\times \nu_2$ is still $G$-invariant and mixing on $X\times\Sigma\times\Sigma$; define finally $\displaystyle{\nu\coloneqq\frac{1}{2}(\nu_1+\nu_2)}$. Part (2) will be proved if we show that $\nu$ is a weak$^*$ limit of a sequence in $M_{G,m}(\Sigma)$.

We now construct such a sequence explicitely. Our proof is very similar to that of \cite[Proposition 2]{GW}.

Put $[e]=\{E\in \Sigma\colon e\in E\}$ where $e$ denotes the identity element of $G$, and, for every $n\geq 1$,
\[
D_n=(B_n\times [e]\times \Sigma)\cup (B_n'\times \Sigma\times [e]).
\]
(We observe that the two subsets are disjoint because $B_n$ and $B_n'$ are.) Define next $\theta_n:X\times \Sigma\times\Sigma\rightarrow \Sigma$ by
\[
\theta_n(x,E_1,E_2)=\{g\colon g^{-1}(x,E_1,E_2)\in D_n\}^+.
\]
For instance, as $g^{-1}(x,E_1,E_2)\in (B_n\times [e]\times\Sigma)$ if and only if $g^{-1}x\in B_n$ and $g\in E_1$, we have in fact:
\[
\theta_n(x,E_1,E_2)=(F(x,n)\cap E_1)^+\cup(F'(x,n)\cap E_2)^+
\]
for all $(x,E_1,E_2)\in X\times \Sigma\times \Sigma$ and $n\geq 1$, where $F(x,n)$ and $F'(x,n)$ are the closed subsets of $G$ as defined in Lemma \ref{lemma2}.

Then we define $\eta_n\in M_{G,m}(\Sigma)$ by $\eta_n=\theta_n(\sigma)$ for every $n$. 

Let $f\in C(\Sigma)$ and $\ep>0$; then $\int fd\eta_n=\int (f\circ\theta_n)d\sigma$ for every $n$, and
we have to show that there exists $N>0$ such that, for every $n>N$ we have
\begin{equation}\label{inegalite}
\Big|\int (f\circ\theta_n)d\sigma-\frac{1}{2}\Big(\int fd\nu_1+\int fd\nu_2\Big)\Big|\leq \ep.    
\end{equation}
Without loss of generality, we assume moreover that $|f(E)|\leq 1$ for every $E\in \Sigma$, and we set $\delta=\ep/7$.

We divide the rest of the proof into 3 steps.

\par\noindent
\textit{Step 1} Lemma \ref{lemma1} implies the existence of $m>0$ such that, for all $E,E'\in \Sigma$ such that $E\cap K_m=E'\cap K_m$ one has $|f(E)-f(E')|\leq\delta$. We observe then that, as $(K_m^+\cap E)\cap K_m=E\cap K_m$ for every $E\in \Sigma$, one has $|f(K_m^+\cap E)-f(E)|\leq\delta$ for every $E$. In particular,
\[
\Big|\int f(K_m^+\cap E)d\nu_j-\int f(E)d\nu_j\Big|\leq \delta
\] 
for $j=1,2$.\\
\textit{Step 2} Set for every $n>m$ : 
\[
\bb=\bigcap_{g\in K_m}gB_n\quad \textrm{and}\quad \bb'=\bigcap_{g\in K_m}gB_n'.
\]
Then Lemma \ref{lemma2} shows that there exists $N>m$ such that 
\[
|\mu(\bb)-1/2|+|\mu(\bb')-1/2|\leq \delta
\]
for every $n\geq N$. Moreover, we have by construction for every $n>N$ and every $x\in \bb$: 
\[
K_m\cap F(x,n)=K_m \quad \textrm{and}\quad K_m\cap F'(x,n)=\emptyset.
\]
Similarly, for every $n>N$ and every $x\in \bb'$, we have
\[
K_m\cap F'(x,n)=K_m \quad \textrm{and}\quad K_m\cap F(x,n)=\emptyset.
\]
We also have for every $n>N$, since $\bb\cap \bb'=\emptyset$ and $|f|\leq 1$,
\begin{multline*}
\Big|\int (f\circ\theta_n)d\sigma-\int_{\bb\times\Sigma\times\Sigma}(f\circ\theta_n)d\sigma-\int_{\bb'\times\Sigma\times\Sigma}(f\circ\theta_n)d\sigma\Big|\\
\leq 
1-\sigma(\bb\times\Sigma\times\Sigma)
-
\sigma(\bb'\times\Sigma\times\Sigma)
\leq
|1/2-\mu(\bb)|+|1/2-\mu(\bb')|\leq \delta.
\end{multline*}
\textit{Step 3} We prove now that, for every $n>N$, one has
\begin{equation}\label{inegfinale}
 \Big| \int_{\bb\times\Sigma\times\Sigma}(f\circ\theta_n)d\sigma-\frac{1}{2}\int f d\nu_1\Big|\leq 3\delta.  
\end{equation}
The proof of the corresponding approximation of $\int fd\nu_2$ is similar.

Fix $n>N$ and $(x,E_1,E_2)\in \bb\times\Sigma\times\Sigma$; then, by Step 2, one has
\[
(K_m\cap F(x,n)\cap E_1)^+\cup (K_m\cap F'(x,n)\cap E_2)^+=K_m^+\cap E_1
\]
and, as $\theta_n(x,E_1,E_2)=(F(x,n)\cap E_1)^+\cup(F'(x,n)\cap E_2)^+$, by 
Step 1, we have
\[
|f(\theta_n(x,E_1,E_2))-f(K_m^+\cap E_1)|\leq \delta.
\]
Thus, we have for every $n>N$:
\begin{multline*}
 \Big| \int_{\bb\times\Sigma\times\Sigma}(f\circ\theta_n)d\sigma-\frac{1}{2}\int f d\nu_1\Big|\\
\leq 
\Big| \int_{\bb\times\Sigma\times\Sigma} (f\circ\theta_n)d\sigma-\int f(K_m^+\cap E)d\mu\times\nu_1\times\nu_2\Big|\\
+
\Big|\mu(\bb)\int f(K_m^+\cap E)d\nu_1-\frac{1}{2}\int f(K_m^+\cap E)d\nu_1\Big|\\
+\frac{1}{2}\Big|\int [f(K_m^+\cap E)-f(E)]d\nu_1\Big|\leq 3\delta
\end{multline*}
Altogether, one has for every $n>N$:
\begin{multline*}
\Big|\int (f\circ\theta_n)d\sigma-\frac{1}{2}\Big(\int fd\nu_1+\int fd\nu_2\Big)\Big|\\
\leq \delta+\Big| \int_{\bb\times\Sigma\times\Sigma}(f\circ\theta_n)d\sigma-\frac{1}{2}\int f d\nu_1\Big|\\
+\Big| \int_{\bb'\times\Sigma\times\Sigma}(f\circ\theta_n)d\sigma-\frac{1}{2}\int f d\nu_2\Big|\leq 7\delta=\ep.
\end{multline*}
The proof of Theorem \ref{HAP1} is complete.
\hfill $\Box$

\section{The case of pairs of tracial von Neumann algebras and their binormal states}

In the present section, $(M,\tau)$ denotes a tracial von Neumann algebra and $B\subset M$ a von Neumann subalgebra of $M$ with the same unit. Most of notations come from \cite{Popa} and \cite{AP}.

We denote by $L^2(M,\tau)$ the Hilbert space obtained by completion of $M$ with respect to the scalar product $(x,y)\mapsto \tau(y^*x)$; the embedding of $M$ into $L^2(M,\tau)$ is denoted by $x\mapsto \hat x$. It is equiped with the order two antilinear isometry $J$ defined by $J\hat x=\wh{x^*}$, $x\in M$.
It is a $M$-bimodule, called the \textit{identity bimodule}, with respect to the left and right actions defined by $x\hat y z=\wh{xyz}$, so that $\hat y z=Jz^*J\hat y$, for all $x,y,z\in M$.

Let $\la M,B\ra$ be the von Neumann subalgebra of $B(L^2(M,\tau))$ generated by $M$ (acting on the left) and by the orthogonal projection $e_B$ of $L^2(M,\tau)$ onto $L^2(B,\tau)$. It is called the \textit{basic construction for} $B\subset M$.

In fact, $e_B$ is the extension of the $\tau$-preserving conditional expectation $E_B$ of $M$ onto $B$. As is well known (e.g. \cite[Subsection 1.3.1]{Popa}), $e_Bxe_B=E_B(x)e_B$ for every $x\in M$, $\Span(Me_BM)$ is a dense *-subalgebra of $\la M,B\ra$ and $e_B\la M,B\ra e_B=Be_B$; in particular, $e_B$ is a finite projection of $\la M,B\ra$. Moreover, $\la M,B\ra=JB'J$, which shows that it is a semifinite von Neumann algebra. Thus, let us denote by $\J(\la M,B\ra)$ the norm-closed ideal of $\la M,B\ra$ generated by the finite projections of $\la M,B\ra$. Thus, an operator $T\in \la M,B\ra$ belongs to $\J(\la M,B\ra)$ if and only if all the spectral projections $1_{[s,\infty)}(|T|)$, $s>0$, are finite projections in $\la M,B\ra$. As recalled above, $e_B\in \J(\la M,B\ra)$.

\subsection{Bimodules, bounded vectors and completely positive maps}

We recall from the introduction that, if $\Hh$ is a $M$-bimodule, a vector $\xi\in \Hh$ is \textit{bounded} if there exists a constant $c>0$ such that $\Vert \xi x\Vert\leq c\Vert x\Vert_2$ and $\Vert x\xi\Vert\leq c\Vert x\Vert_2$. The set of all bounded vectors in $\Hh$ is denoted by $\Hh_0$, and, for $\xi\in \Hh_0$, we denote by $L_\xi:L^2(M,\tau)\rightarrow \Hh$ the bounded operator given by $L_\xi \hat y=\xi y$, $y\in M$.

A \textit{pointed $(B\subset M)$-bimodule} is a pair $(\Hh,\xi)$, where $\xi\in \Hh$, $\Vert \xi\Vert=1$, and $\xi$ is $B$\textit{-central}, which means that $b\xi=\xi b$ for every $b\in B$. We denote by $\Hh(B)$ the subspace of all $B$-central vectors of $\Hh$ and $\Hh_0(B)$ the set of all bounded, $B$-central vectors in $\Hh$. If $B=\C 1$, we simply call $(\Hh,\xi)$ a \textit{pointed $M$-bimodule}.

Denote by $\CP(M)$ the set of normal, completely positive maps $\phi:M\rightarrow M$, and by $\CP(B\subset M)$ the subset of $B$\textit{-bimodular} ones, i.e. elements $\phi\in \CP(M)$ that satisfy  $\phi(b_1xb_2)=b_1\phi(x)b_2$ for all $x\in M,b_1,b_2\in B$. For instance, the identity map $\id_M: x\mapsto x$ and the conditional expectation $E_B:M\rightarrow B\subset M$ are $B$-bimodular.

Let $\phi\in \CP(M)$ be normalized so that $\tau(\phi(1))=1$. There exists a pointed $M$-bimodule $(\Hh_\phi,\xi_\phi)$ with the following properties \cite[Subsection 1.1.2]{Popa}:
\begin{itemize}
\item $\xi_\phi\in \Hh_\phi$ is a \textit{cyclic} vector, i.e. $\Span(M\xi_\phi M)$ is dense in $\Hh_\phi$;
\item one has $\la x\xi_\phi y|\xi_\phi\ra=\tau(\phi(x)y)$ for all $x,y\in M$; in particular, $\Vert \xi_\phi\Vert^2=\tau(\phi(1))=1$;
\item $\phi$ is $B$-bimodular if and only if $\xi_\phi$ is $B$-central.
\end{itemize}

Next, let us denote by $\CP_\tau(M)$ (resp. $\CP_\tau(B\subset M)$) the set of elements $\phi\in \CP(M)$ (resp. $\CP(B\subset M)$) for which there exists $c>0$ such that $\tau\circ\phi\leq c\tau$ and $\tau(\phi(1))=1$.
Thus, $(\Hh_\phi,\xi_\phi)$ is a pointed $M$-bimodule or a pointed $(B\subset M)$-bimodule depending on whether $\phi\in\CP_\tau(M)$ or $\phi\in \CP_\tau(B\subset M)$.

Every $\phi\in\CP(M)$ for which there exists $c>0$ such that $\tau\circ\phi\leq c\tau$
extends to a bounded operator $\hat\phi\in B(L^2(M,\tau))$ by the formula $\hat\phi(\hat x)=\wh{\phi(x)}$, $x\in M$. It is the case for every element of $\CP_\tau(M)$, and 
if $\phi\in \CP_\tau(B\subset M)$ then one checks that $\hat \phi\in B'\cap \la M,B\ra$ \cite[Lemma 1.2.1]{Popa}.

In fact, the reason of our choice of the above definition of $\CP_\tau(B\subset M)$ is because it is exactly the set of completely positive maps associated with the binormal states whose cyclic vectors are $B$-central and bounded that will be defined in Subsection 4.2 below.

We gather the main properties of bounded vectors in $(B\subset M)$-bimodules in the following proposition whose most part for the case $B=\C 1$ is in \cite[Chapter 13]{AP}. We provide a proof for the reader's convenience.

\begin{proposition}\label{boundedv}
Let $\Hh$ be a $M$-bimodule and $B\subset M$ a von Neumann subalgebra of $M$.
\begin{enumerate}
\item [(1)] The subset $\Hh_0(B)$ of bounded, $B$-central vectors is a dense subspace of $\Hh(B)$; we have moreover $(B'\cap M)\Hh_0(B) (B'\cap M)=\Hh_0(B)$, and for every $\xi\in \Hh_0(B)$ and all $a,b\in B'\cap M$, we have $L_{a\xi b}=aL_\xi b$.
\item [(2)] For all $\xi,\eta\in \Hh_0(B)$ and every $x\in M$ then $L_\eta^* xL_\xi\in M$ and the map $x\mapsto L_\eta^*xL_\xi$ is $B$-bimodular. In particular, the map $\phi_\xi:x\mapsto \phi_\xi(x)\coloneqq L_\xi^* xL_\xi$ belongs to $\CP(B\subset M)$.
\item [(3)] Let $\xi\in \Hh_0(B)$ and let $c>0$ be such that $\Vert x\xi\Vert\leq c\Vert x\Vert_2$ for every $x\in M$. Then the completely positive map $\phi_\xi$ satisfies the following inequality: 
\[
\tau\circ\phi_\xi(x)\leq c^2\tau(x)\quad (x\in M_+).
\]
In particular, if $\xi$ is a unit vector, then $\phi_\xi\in \CP_\tau(B\subset M)$..
\item [(4)] Let $\xi\in \Hh_0(B)$. If $a,b\in B'\cap M$ then
\[
\wh{\phi_{a\xi b}}=b^*Jb^*J\wh{\phi_\xi}a^*Ja^*J.
\]
\end{enumerate}
\end{proposition}
\textsc{Proof.} (1) Fix $\xi\in \Hh(B)$, and let $S,T\in L^1(M,\tau)_+$ be the Radon-Nikodym derivatives of the states
\[
\la x\xi|\xi\ra=\tau(xT)\quad \textrm{and}\quad \la \xi y|\xi\ra=\tau(yS)\quad (x,y\in M).
\]
Fix an element $x\in M$ and a unitary operator $u\in U(B)$. Then one has
\[
\tau(xT)=\la xu\xi u^*|\xi\ra=\la xu\xi|\xi u\ra=\la xu\xi|u\xi\ra=
\la u^*xu\xi|\xi\ra=\tau(u^*xuT)=\tau(xuTu^*)
\]
which shows that $T\in B'\cap L^1(M,\tau)_+$, and the same holds true for $S$.

For every positive integer $n$, set $p_n=1_{[0,n]}(T), q_n=1_{[0,n]}(S)\in B'\cap M$, so that $p_n,q_n\nearrow 1$ strongly and $p_nTp_n, q_nSq_n\leq n$. Set then $\xi_n=p_n\xi q_n$. Then we have
\begin{multline*}
\Vert x\xi_n\Vert^2  
\leq\Vert xp_n\xi\Vert^2=\la p_nx^*xp_n\xi|\xi\ra\\
=\tau(p_nx^*xp_nT)=\tau(x^*xp_nTp_n)\leq n\Vert x\Vert_2^2
\quad (x\in M)
\end{multline*}
and similarly $\Vert \xi_ny\Vert^2\leq n\Vert y\Vert_2^2$, $y\in M$.

As $\Vert \xi-\xi_n\Vert \leq \Vert (1-p_n)\xi\Vert+\Vert \xi(1-q_n)\Vert$, this shows that $\Hh_0(B)$ is dense in $\Hh(B)$. Equalities $(B'\cap M)\Hh_0(B)(B'\cap M)=\Hh_0(B)$ and $L_{a\xi b}=aL_\xi b$ are straightforward.\\
(2) For $\xi,\eta\in \Hh_0(B)$ and $x,y,z,w\in M$,
\begin{multline*}
\la Jz^*JL^*_\eta xL_\xi\hat y|\hat w\ra=\la L^*_\eta xL_\xi\hat y|\wh{wz^*}\ra\\
=\la x\xi y|\eta(wz^*)\ra
=\la x\xi(yz)|\eta w\ra=
\la L^*_\eta xL_\xi Jz^*J\hat y|\hat w\ra,
\end{multline*}
which shows that $L^*_\eta xL_\xi\in M$. Moreover, it is obvious that $L_\xi b=bL_\xi$ for every $b\in B$, and this ends the proof of (2).\\
(3) Let $\xi\in\Hh_0(B)$ and $c>0$ such that $\Vert x\xi\Vert\leq c\Vert x\Vert_2$ for every $x\in M$. Then we get for $x\in M_+$:
\[
\tau\circ\phi_\xi(x)=\la L^*_\xi xL_\xi\hat 1|\hat 1\ra=\la x\xi|\xi\ra=\Vert x^{1/2}\xi\Vert^2
\leq c^2\Vert x^{1/2}\Vert_2^2=c^2\tau(x),
\]
(4) By (1), one has $L_{a\xi b}=aL_\xi b$ and $L^*_{a\xi b}=b^*L^*_\xi a^*$ and thus
\[
\phi_{a\xi b}(x)=b^*L^*_\xi a^*xaL_\xi b=b^*\phi_\xi (a^*xa)b\quad (x\in M)
\]
and
\[
\wh{\phi_{a\xi b}}(\hat x)=\wh{b^*\phi_\xi(a^*xa)b}=b^*Jb^*J\wh{\phi_\xi}(a^*Ja^*J\hat x)
\]
for every $x\in M$.
This ends the proof of Proposition \ref{boundedv}.
\hfill $\Box$

We still need to recall the definition of composition of $M$-bimodules. See for instance \cite[Proposition 13.2.1]{AP}. Let $\Hh$ and $\K$ be such $M$-bimodules. We denote by $\Hh\otimes_M \K$ the $M$-bimodule obtained from the algebraic relative tensor product $\Hh_0\odot_M\K$ equipped with the scalar product given by
\[
\la \xi_1\otimes_M\eta_1|\xi_2\otimes_M\eta_2\ra\coloneqq \la L^*_{\xi_2}L_{\xi_1}\eta_1|\eta_2\ra \quad (\xi_j\in \Hh_0, \eta_j\in\K).
\]
It is gifted with the left and right actions of $M$ given by $x(\xi\otimes_M \eta)y=(x\xi)\otimes_M(\eta y)$.

\begin{lemma}\label{tensor}
Let $\Hh$ and $\K$ be $M$-bimodules and let $\xi\in\Hh_0(B)$ and $\eta\in\K_0(B)$. Then $\xi\otimes_M\eta\in (\Hh\otimes_M\K)_0(B)$ and
\[
\phi_{\xi\otimes_M\eta}=\phi_\eta\circ \phi_\xi.
\]
\end{lemma}
\begin{proof}
Since $\Hh\otimes_M\K$ is a relative tensor product over $M$, one has for every $x\in M$: $(\xi x)\otimes_M\eta=\xi\otimes_M (x\eta)$. In particular, this implies that $b(\xi\otimes_M\eta)=(\xi\otimes_M\eta)b$ for every $b\in B$.

Let now $c>0$ be large enough so that 
\[
\Vert x\xi\Vert,\Vert x\eta\Vert\leq c\Vert x\Vert_2\quad \textrm{and}\quad
\Vert \xi y\Vert,\Vert \eta y\Vert\leq c\Vert y\Vert_2
\]
for all $x,y\in M$.
\\
Fix $x\in M$ first; one has 
\begin{align*}
\Vert x(\xi\otimes_M \eta)\Vert^2
&=
\la (x\xi)\otimes_M \eta| (x\xi)\otimes_M \eta\ra=\la L_{x\xi}^*L_{x\xi}\eta|\eta\ra\\
&=
\la L^*_\xi x^*xL_\xi\eta|\eta\ra=\Vert \phi_\xi(x^*x)^{1/2}\eta\Vert^2\leq c^2\Vert \phi_\xi(x^*x)^{1/2}\Vert_2^2\\
&=
c^2\tau(\phi_\xi(x^*x))=c^2\la L_\xi^*x^*xL_\xi \hat 1|\hat 1\ra\\
&=
c^2 \la x^*x\xi|\xi\ra\leq c^4\Vert x\Vert_2^2.
\end{align*}
Next, for $y\in M$ one has 
\begin{align*}
\Vert (\xi\otimes_M\eta)y\Vert^2
&=
\Vert \xi\otimes_M(\eta y)\Vert^2=\la L_\xi^*L_\xi \eta y|\eta y\ra\\
&\leq
c^2 \Vert \eta y\Vert^2\leq c^4 \Vert y\Vert_2^2.
\end{align*}
This proves that $\xi\otimes_M \eta\in (\Hh\otimes_M\K)_0(B)$. \\
Finally, if $x,y,z\in M$, one has 
\begin{align*}
\la \phi_{\xi\otimes_M\eta}(x)\hat y|\hat z\ra
&=
\la L^*_{\xi\otimes_M\eta}xL_{\xi\otimes_M\eta} \hat y|\hat z\ra=\la (x\xi)\otimes_M(\eta y)|\xi\otimes_M(\eta z)\ra\\
&=
\la L_\xi^*L_{x\xi}\eta y|\eta z\ra=\la L^*_\eta L^*_\xi xL_\xi L_\eta \hat y|\hat z\ra\\
&=
\la \phi_\eta\circ\phi_\xi (x)\hat y|\hat z\ra
\end{align*}
This ends the proof of the present lemma.
\end{proof}

We end the present subsection with the following definition (see \cite[Section 2]{Popa}).

\begin{definition}\label{relHAP}
Let $(M,\tau)$ be a tracial von Neumann algebra and $B\subset M$ a von Neumann subalgebra of $M$.
\begin{enumerate}
\item [(1)] Let $\phi\in \CP(B\subset M)$ which satisfies the following property: there exists $c>0$ such that $\tau\circ\phi\leq c\tau$, so the associated operator $\hat \phi$ on $L^2(M,\tau)$ is bounded. Then we say that $\phi$ is \textit{$B$-compact} if $\hat\phi\in \J(\la M,B\ra)$.
\item [(2)] We say that $M$ has \textit{property} H \textit{relative to} $B$ if there exists a net of maps $(\phi_i)_{i\in I}\subset \CP(B\subset M)$ that has the following properties:
\begin{enumerate}
\item for every $i\in I$, $\phi_i$ is \textit{subtracial}: $\tau\circ\phi_i\leq \tau$;
\item for every $i\in I$, $\phi_i$ is $B$-compact;
\item for every $x\in M$, one has $\lim_i\Vert \phi_i(x)-x\Vert_2=0$.
\end{enumerate}
\end{enumerate}
\end{definition}

In the special case $B=\C 1$, which corresponds to the Haagerup property for finite von Neumann algebras, it is shown in \cite[Proposition 4]{Jol} that it doen't depend on the trace $\tau$. As mentioned by J. Bannon and J. Fang in \cite[Section 2]{BF}, using \cite[Proposition 2.4.2]{Popa} and an adaptation of \cite[Proposition 4]{Jol} shows that property H relative to $B$ is independent of the trace $\tau$.

\subsection{$B$-mixing binormal states}

Following \cite{EL}, a \textit{binormal state} on $M$ is a linear functional $\f$ on the algebraic tensor product $M\otimes M^\op$ such that:
\begin{itemize}
\item for every $X\in M\otimes M^\op$, one has $\f(X^*X)\geq 0$ and $\f(1\otimes 1^\op)=1$;
\item $\f$ is separately normal: for every $y\in M$, the functionals $x\mapsto \f(x\otimes y^\op)$ and $x\mapsto \f(y\otimes x^\op)$ are normal.
\end{itemize}
We denote here by $\bin(M)$ the convex set of all binormal states on $M$. It is a convex, weak*-dense subset of the state space of the $C^*$-algebra $M\otimes_\bin M^\op$, the latter being the completion of $M\otimes M^\op$ with respect to the $C^*$-norm
\[
\Vert X\Vert_\bin\coloneqq \sup\{\Vert \pi_\f(X)\Vert\colon \f\in\bin(M)\}
\]
where $(\Hh_\f,\pi_\f,\xi_\f)$ denotes the GNS triple associated to $\f$: $\pi_\f$ is a $*$-representation of $M\otimes M^\op$ on the Hilbert space $\Hh_\f$ and $\xi_\f\in \Hh_\f$ is a cyclic and unit vector such that $\f(X)=\la \pi_\f(X)\xi_\f|\xi_\f\ra$ for every $X\in M\otimes M^\op$.

Observe that every $\f\in \bin(M)$ gives rise to the pointed $M$-bimodule $\Hh_\f$ by the following formula:
\[
\la x\xi_\f y|\xi_\f\ra=\f(x\otimes y^\op)\quad (x,y\in M)
\]
and conversely, if $(\Hh,\xi)$ is a pointed $M$-bimodule, we get the following element $\f_\xi$ of $\bin(M)$:
\[
\f_\xi(x\otimes y^\op)=\la x\xi y|\xi\ra\quad (x,y\in M).
\]
Given $\f\in\bin(M)$, it is straightforward to verify that $\xi_\f$ is $B$-central if and only if one has for all $x,y\in M$ and $b\in B$:
\[
\f(xb\otimes y^\op)=\f(x\otimes (by)^\op)=\f(x\otimes y^\op b^\op).
\]
We denote by $\bin(B\subset M)$ the (convex) subset of all elements of $\bin(M)$ whose associated unit vector is $B$-central.

Finally, let $\f\in \bin(B\subset M)$; we say that $\f$ is \textit{$B$-mixing} if its associated vector $\xi_\f\in \Hh_\f$ is bounded and if the corresponding cp map $\phi_{\xi_\f}$ belongs to $\CP_\tau(B\subset M)$ and is $B$-compact. We denote by $\bin_\mix(B\subset M)$ the set of all $B$-mixing binormal states in $\bin(B\subset M)$.

\subsection{Proof of Theorem \ref{HAP2}}

We are ready to state and prove the following relative version of Theorem \ref{HAP2}.

\begin{theorem}
Let $(M,\tau)$ be a tracial von Neumann algebra and let $B\subset M$ be a von Neumann subalgebra that has the same unit as $M$. Then $M$ has property $\mathrm H$ relative to $B$ if and only if $\bin_\mix(B\subset M)$ is weak*-dense in $\bin(B\subset M)$.
\end{theorem}
\textsc{Proof.} 
Assume that $M$ has property H relative to $B$. By \cite[Theorem 2.2]{BF}, there exists a net $(\psi_\be)_{\be\in \Gm} \subset \CP_\tau(B\subset M)$ such that, furthermore, $\psi_\be(1)=1$, $\tau\circ \psi_\be=\tau$ and $\wh{\psi_\be}\in \J(\la M,B\ra)$ for every $\be\in \Gm$, and such that
\[
\lim_{\be\in\Gm}\Vert \psi_\be(x)-x\Vert_2=0\quad (x\in M).
\]
Let then $\f\in \bin(B\subset M)$, a finite set $F\subset M$ and $\ep>0$. By Proposition \ref{boundedv}, there exists a unit vector $\xi\in (\Hh_\f)_0(B)$ such that
\[
\Vert \xi_\f-\xi\Vert\leq \frac{\ep}{4(1+\max_{x,y\in F}\Vert x\Vert\Vert y\Vert)}.
\]
There exists some $\be\in \Gm$ such that $\phi\coloneqq \psi_\be$ belongs to $\CP_\tau(B\subset M)$ with $\hat\phi\in \J(\la M,B\ra)$, and
\[
\max_{x\in F}\Vert \phi(\phi_\xi(x))-\phi_\xi(x)\Vert_2<\frac{\ep}{2(1+\max_{y\in F}\Vert y\Vert)}.
\]
By Lemma \ref{tensor}, $\phi\circ\phi_\xi=\phi_{\xi\otimes_M\xi_\phi}$ is $B$-bimodular and $B$-compact. Let then $\psi\in \bin_\mix(B\subset M)$ be defined by
\[
\psi(x\otimes y^\op)=\tau(\phi_{\xi\otimes_M\xi_\phi}(x)y)=\la x(\xi\otimes_M\xi_\phi)y|\xi\otimes_M\xi_\phi\ra
\quad (x,y\in M).
\]
One has for all $x,y\in F$:
\begin{align*}
|\f(x\otimes y^\op)-\psi(x\otimes y^\op)|
&=
|\la x\xi_\f y|\xi_\f\ra-\tau(\phi\circ\phi_\xi(x)y)|\\
&\leq
|\la x\xi_\f y|\xi_\f\ra-\la x\xi y|\xi\ra|+|\tau(\phi_\xi(x)y)-\tau(\phi\circ\phi_\xi(x)y)|\\
&\leq
\frac{\ep}{2}+|\la x\xi_\f y|\xi_\f\ra-\la x\xi y|\xi\ra|\\
&\leq
\frac{\ep}{2}+|\la x(\xi_\f-\xi)y|\xi_\f\ra|+|\la x\xi y|\xi_\f-\xi\ra|\\ 
&\leq
\frac{\ep}{2}+2\Vert x\Vert \Vert y\Vert \Vert \xi_\f-\xi\Vert\leq \ep.
\end{align*}
This shows that $\bin_\mix(B\subset M)$ is dense in $\bin(B\subset M)$.\\
Conversely, suppose that $\bin_\mix(B\subset M)$ is dense in $\bin(B\subset M)$ and let $\f_\tau\in \bin(B\subset M)$ be defined by $\f_\tau(x\otimes y^\op)=\tau(xy)$ for all $x,y\in M$ as at the end of Section 1. This means that $\phi_{\f_\tau}=\id_M$.
By hypothesis, there exists a net $(\phi_i)_{i\in I}\subset \CP_\tau(B\subset M)$ such that each $\phi_i$ is $B$-compact and
\[
\lim_{i\in I}\tau(\phi_i(x)y)=\tau(xy)\quad (x,y\in M).
\]
Set then $\xi_i=\xi_{\phi_i}$ for every $i\in I$. One 
has for every $x\in M$:
\begin{align*}
\Vert x\xi_i-\xi_i x\Vert^2
&=
\la x^*x\xi_i|\xi_i\ra+\la \xi_i xx^*|\xi_i\ra-2\re\la x\xi_i|\xi x\ra\\
&=
\tau(\phi_i(x^*x))+\tau(\phi_i(1)xx^*)-2\re\tau(\phi_i(x)x)\\
&\to_{i\in I}
\tau(x^*x)+\tau(xx^*)-2\re\tau(xx^*)=0
\end{align*}
and 
\[
\lim_{i\in I}\la x\xi_i|\xi\ra=\lim_{i\in I}\la \xi_i x|\xi_i\ra=\tau(x).
\]
By \cite[Lemma 13.3.11]{AP}, replacing the $\xi_i$'s by some of their suitable finite convex combinations first, and then perturbating the latter as in \cite[Lemma 13.1.11]{AP}, we assume that the $\phi_i$'s are furthermore subtracial and satisfy:
\[
\lim_{i\in I}\Vert \la\cdot\ \xi_i|\xi_i\ra-\tau\Vert=
\lim_{i\in I}\Vert \la \xi_i\cdot|\xi_i\ra-\tau\Vert=0.
\]
We observe that, as in the proof of Proposition \ref{boundedv}, this process will affect neither $B$-bimodularity nor $B$-compactness of the $\phi_i$'s.

Then, given a finite set $F\subset M$ and $\ep>0$, set $F_1=F\cup F^*\cup\{x^*x\colon x\in F\cup F^*\}$
and
let us choose $i\in I$ so that
\[
\max_{x\in F_1}\{|\tau(x)-\la x\xi_i|\xi_i\ra|+|\tau(x)-\la \xi_i x|\xi_i\ra|,
\Vert x\xi_i-\xi_i x\Vert\}\leq \frac{\ep}{2(1+\max_{y\in F}\Vert y\Vert)}.
\]
As in the proof of \cite[Proposition 13.3.10]{AP}, we get for every $x\in F$, since $\phi_i(1)\leq 1$, using Schwarz inequality:
\begin{align*}
\Vert \phi_i(x)-x\Vert_2^2
={} &
\tau(\phi_i(x^*)\phi_i(x))+\tau(x^*x)-\tau(\phi_i(x^*)x)-\tau(\phi_i(x)x^*)
\\
\leq {} &
2\tau(x^*x)-\tau(\phi_i(x^*)x)-\tau(\phi_i(x)x^*)\\
={} &
(\tau(x^*x)-\la x^*\xi_i x|\xi_i\ra))+(\tau(xx^*)-\la x\xi_i x^*|\xi_i\ra)\\
\leq {} &
|\tau(x^*x)-\la x^*x\xi_i|\xi_i\ra|+|\tau(xx^*)-\la \xi_i xx^*|\xi_i\ra|\\
 &
+2\Vert x\Vert\Vert x\xi_i-\xi_i x\Vert\\
\leq {} & \ep.
\end{align*}
This ends the proof of Theorem \ref{HAP2}
\hfill $\Box$

\section{Proof of Theorem \ref{propT}}

In this final section, $M$ denotes a type II$_1$ factor. Recall for instance from \cite[Theorem 14.5.2]{AP} that $M$ has \textit{property (T)} if there exists a finite set $F\subset M$ and $\delta>0$ such that if $(\Hh,\xi)$ is a pointed $M$-bimodule satisfying
\[
\max_{x\in F}\| x\xi-\xi x\|\leq \delta,
\]
then $\Hh$ contains a non-zero $M$-central vector $\eta$, i.e. $x\eta=\eta x$ for every $x\in M$.

Let us now turn to the proof of Theorem \ref{propT}.

Assume that $M$ has property (T); by \cite[Theorem 5.1]{JolT}, for every net $(\f_i)_{i\in I}\subset \bin(M)$ which converges weakly$^*$ to $\f_\tau$, one has
\[
\sup_{u\in U(M)}|\f_i(u\otimes u^{*\op})-1|\to_i 0.
\]
Set $\Hh_i=\Hh_{\f_i}$ and $\xi_i=\xi_{\f_i}$ for short, and assume by contradiction that there is a net $(\f_i)_{i\in I}\subset \bin_\wmix(M)$ which converges weakly$^*$ to $\f_\tau$. This means that
\[
\sup_{u\in U(M)}|\la u\xi_i u^*|\xi_i\ra-1|\to_i 0.
\]
Then choose $i\in I$ such that 
\[
\sup_{u\in U(M)}|\la u\xi_i u^*|\xi_i\ra-1|<\frac{1}{2}.
\]
This implies that $\re\la u\xi_iu^*|\xi_i\ra\geq 1/2$ for every $u\in U(M)$. Let next $K_i$ be the norm-closed convex hull of $\{u\xi_i u^*\colon u\in U(M)\}$ in the unit ball of $\Hh_i$. Then one has $\re\la \zeta|\xi_i\ra\geq 1/2$ for every $\zeta\in K_i$, and in particular, if $\eta\in K_i$ is the element of minimal norm, we have $\re\la\eta|\xi_i\ra\geq 1/2$. Thus it is non-zero, and its uniqueness implies that $u\eta u^*=\eta$ for every $u\in U(M)$. But this contradicts the fact that $\f_i$ is weakly mixing as the associated unitary representation $\rho_i$ of $U(M)$ would contain the trivial representation.

Conversely, let us assume that $M$ does not have property (T). By Theorems 3.6 and 3.10 in \cite{JolWMAP}, there exists a net of pointed $M$-bimodules $((\Hh_i,\xi_i))_{i\in I}$ with the following properties:
\begin{enumerate}
\item [(a)] $\lim_i (\|\la\cdot\ \xi_i|\xi_i\ra-\tau\|+\|\la \xi_i\cdot|\xi_i\ra-\tau\|)=0$;
\item [(b)] $\lim_j \|x\xi_i-\xi_i x\|=0$ for every $x\in M$;
\item [(c)] the unitary representation $\pi_i:U(M)\rightarrow U(\Hh_i)$ defined by $\pi_i(u)\xi=u\xi u^*, u\in U(M),\xi\in\Hh_i$, is weakly mixing.
\end{enumerate}
Let $\f_i$ be the binormal state associated to $(\Hh_i,\xi_i)$. Then it belongs to $\bin_\wmix(M)$, and one has for all $x,y\in M$
\begin{align*}
|\f_i(x\otimes y^\op)-\f_\tau(x\otimes y^\op)|
&=
|\la x\xi_i y|\xi_i\ra-\tau(xy)|\\
&=
|\la x(\xi_iy-y\xi_i)|\xi_i\ra+\la xy\xi_i|\xi_i\ra-\tau(xy)|\\
&\leq
\|x\|\|\xi_i y-y\xi_i\|+\|x\|\|y\|\|\la \cdot\xi_i|\xi_i\ra-\tau\|.
\end{align*}
Thus $\lim_i|\f_i(x\otimes y^\op)-\f_\tau(x\otimes y^\op)|=0$. This proves that $\f_\tau\in \ol{\bin_\wmix(M)}^{w*}$ , and this ends the proof of Theorem \ref{propT}.
\hfill $\Box$

\medskip\noindent

\textbf{Acknowledgements.} I am very greatful to E. Glasner for having pointed out his result on Theorem \ref{HAP1}, to I. Vigdorovich, A. Valette and  the referee for their useful and pertinent comments.

\par\vspace{1cm}

\bibliographystyle{plain}
\bibliography{refFixed2}

\end{document}